\UseRawInputEncoding
\documentclass[12pt]{amsart}

\usepackage{a4, a4wide}
\usepackage{amssymb}
\usepackage{amsmath, bm}
\usepackage{amsthm}
\usepackage{amstext}
\usepackage{amscd}
\usepackage{latexsym}
\usepackage{graphics}
\usepackage{color}
\usepackage{enumitem}
\usepackage[all, cmtip]{xy}
\usepackage{stmaryrd}
\usepackage[colorlinks, pagebackref]{hyperref}
\hypersetup{
  colorlinks=true,
  citecolor=blue,
  linkcolor=blue,
  urlcolor=blue}

\makeatletter
\def\@tocline#1#2#3#4#5#6#7{\relax
  \ifnum #1>\c@tocdepth 
  \else
    \par \addpenalty\@secpenalty\addvspace{#2}%
    \begingroup \hyphenpenalty\@M
    \@ifempty{#4}{%
      \@tempdima\csname r@tocindent\number#1\endcsname\relax
    }{%
      \@tempdima#4\relax
    }%
    \parindent\z@ \leftskip#3\relax \advance\leftskip\@tempdima\relax
    \rightskip\@pnumwidth plus4em \parfillskip-\@pnumwidth
    #5\leavevmode\hskip-\@tempdima
      \ifcase #1
      \or\or \hskip 2em \or \hskip 2homologyem \else \hskip 3em \fi%
      #6\nobreak\relax
    \dotfill\hbox to\@pnumwidth{\@tocpagenum{#7}}\par
    \nobreak
    \endgroup
  \fi}
\makeatother

\theoremstyle{plain}

\newtheorem{theorem}{Theorem}[section]

\newtheorem{lemma}[theorem]{Lemma}

\newtheorem{proposition}[theorem]{Proposition}

\theoremstyle{definition}

\newtheorem{remark}[theorem]{Remark}

\newtheorem{definition}[theorem]{Definition}

\numberwithin{equation}{section}

\newcommand{\dlim}{\mathop{\varinjlim}\limits}  

\newcommand{\inj}{\hookrightarrow}

\newcommand{\Spec}{{\rm Spec \,}}

\renewcommand{\tilde}{\widetilde}

\newcommand{\Z}{{\mathbb Z}}

\newcommand{\KM}{{K}^{\rm M}}                      

\newcommand{\A}{\mathbb A}

\newcommand{\calK}{\mathcal K}

\newcommand{\sO}{\mathcal O}

\def\<{\langle}
\def\>{\rangle} 
\def\-{\overline} 
\def\~{\widetilde}
\def\^{\widehat}
\def\fr{\mathfrak}
\def\@{\mathcal}
\def\!{\mathscr}
\def\#{\mathbb}
\def\&{\mathbf}
\def\_{\underline}
\def\Dot{\bullet}

\input{xy}
\xyoption{all}

\begin{document}

\title{A remark on Gersten complex for Milnor $K$-theory}

\author{Rakesh Pawar}
\address{Department of Mathematics, Indian Institute of Science and Education Research (IISER), Mohali, Knowledge city, Sector 81, Manauli, PO, Sahibzada Ajit Singh Nagar, Punjab-140306, India}
\email{rakesh.pawar.math@gmail.com}
\address{Department of Mathematical Sciences, Indian Institute of Science Education and Research Pune, Dr. Homi Bhabha Road, Pashan, Pune 411008, India.}
\email{ rakesh.pawar@acads.iiserpune.ac.in}
\date{\today}


\begin{abstract} In this note, we consider the Gersten complex for Milnor $K$-theory over a regular local Henselian domain $S$ and prove that in degrees $\geq \dim S\geq 1$, the Gersten complex of an essentially smooth Henselian local $S$-scheme is exact.
\end{abstract}

\maketitle

\section{Introduction}

Let $X$ be an excellent scheme as in~\cite[\S 7, 7.8]{EGA4}. In~\cite{kato86} Kato considered the sequence of abelian groups 
\begin{equation}\label{GC}
\bigoplus_{\eta\in X^{(0)}}\KM_n(\kappa(\eta))\to \bigoplus_{x\in X^{(1)}}\KM_{n-1}(\kappa(x))\to \cdots \to \bigoplus_{x\in X^{(d)}}\KM_{n-d}(\kappa(x))\to\cdots\to 0
\end{equation}
for Milnor K-theory (similar to the Gersten complex for Quillen's algebraic K-theory), and showed that for $X$, this sequence of abelian groups is a complex~\cite[Proposition 1]{kato86}.  Here $X^{(p)}$ denotes the set of codimension $p$ points of $X$. 

According to the Gersten conjecture for Milnor K-theory, for any Noetherian regular local ring $A$ of dimension $d$, the complex (which we refer to as \textit{Gersten complex})  
\begin{equation*}
0\to\KM_n(A)\to \KM_n(Frac(A))\to \bigoplus_{\fr p\in \Spec A^{(1)}}\KM_{n-1}(\kappa(\fr p))\to \cdots \to \bigoplus_{\fr p\in \Spec A^{(d)}}\KM_{n-d}(\kappa(\fr p))\to 0
\end{equation*}
is exact, for each $n\geq 0$. Here $Frac(A)$ denotes the field of fractions of $A$ and $\Spec A^{(p)}$ denotes the set of codimension $p$ points of $\Spec A$. 
For $A$ containing a field, Gersten conjecture has been verified by M. Kerz~\cite{kerz09, kerz10}. More  recently, M. L\"uders proved the following relative version.
\begin{theorem}(\cite[Theorem 1.1]{lueders20})
Let $X=\Spec A$ be the spectrum of a smooth local ring $A$ over an excellent discrete valuation ring, then the complex
\begin{equation*}
 \KM_n(A)\to\KM_n(Frac(X))\to \bigoplus_{x\in X^{(1)}}\KM_{n-1}(\kappa(x))\to \cdots \to \bigoplus_{x\in X^{(d)}}\KM_{n-d}(\kappa(x))\to 0. \end{equation*}
is exact in degrees $>0$ for each $n\geq 0$.
Furthermore, it is exact and the first map is injective if the exactness on the left holds for all discrete valuation rings.
\end{theorem}
In this note, we consider the complex~\eqref{GC}, for an excellent scheme $X$ over a base $S=\Spec R$, the spectrum of a regular local Henselian domain $R$. We show that a certain tail (depending on the Krull dimension of $S$) of the Gersten complex is exact. More precisely, we prove the following result.

\begin{theorem}(see Theorem~\ref{main})\label{mainthm}
Let $S=\Spec R$ for $R$ a regular local Henselian domain of Krull dimension $\geq 1$.  
Let $X$ be an essentially smooth Henselian local $S$-scheme of dimension $d$,  then
\begin{equation*} A^p(\KM_n)(X)=0
\end{equation*}
\text{for} $p\geq \dim S\geq 1$ and $n\geq 0.$ More explicitly,
\begin{equation*}
\KM_n(Frac(X))\to \bigoplus_{x\in X^{(1)}}\KM_{n-1}(\kappa(x))\to \cdots \to \bigoplus_{x\in X^{(d)}}\KM_{n-d}(\kappa(x))\to 0
\end{equation*}
is exact in degrees $p\geq \dim S$.
\end{theorem}

We note that when $S$ is of dimension one, we recover M. L\"uders' result~\cite[Theorem 3.1]{lueders20}. The result in \textit{loc. cit.} makes crucial use of the fact that the base is of dimension one. The main contribution of this paper is to realise that Gabber's presentation lemma can be used to by-pass certain results special to dimension one. 

\begin{remark}
We recall some related results in the literature to put our results in the historical context. 
\begin{enumerate}

\item In case the base $S$ is the spectrum of a field, then for a regular local ring over a field, the Gersten conjecture holds for Motivic cohomology~\cite{VSF}, Higher Chow groups~\cite{Bloch Higher Chow groups}, Algebraic K-theory~\cite{Quillen}, and \'etale cohomology~\cite{Bloch-Ogus}. All these results can be extended to equi-characteristic regular local rings by using the techniques of~\cite{Panin}.
\item  The Gersten conjecture holds for Motivic cohomology with finite coefficients for smooth schemes over a discrete valuation ring~\cite[Corollary 4.3]{Geis04}. The Gersten conjecture holds for Algebraic K-theory with finite coefficients~\cite[Theorem 8.2]{Geis-Lev}. Using comparison with Algebraic K-theory, there is a weaker result: for $m$ coprime with $(n-1)!$ (for example $m=p^r$ and $p> n-1$), the Gersten conjecture holds for $\^{\calK}^{M}_n/m$~\cite[Proposition 3.2]{lueders19}, where $\^{\calK}^{M}_n$ denotes the improved Milnor K-theory. This result has been improved in~\cite[Theorem 4.1]{mor-lud} to show that the Gersten conjecture holds for $\^{\calK}^{M}_n/m$, for $m=p^r$, for $p$-Henselian, ind-smooth algebras over a complete discrete valuation ring. 
\end{enumerate}

\end{remark}



\section*{Acknowledgements}
 The author would like to thank Amit Hogadi and Chetan Balwe for constant discussions and input on the earlier drafts of this paper. The author would like to thank the anonymous referee for the comments and suggestions which improved the exposition. Work on this paper began when the author was supported by the Postdoctoral Fellowship position in Indian Institute of Science Education and Research (IISER), Mohali and during its completion the author was supported by IISER Pune and NBHM Postdoctoral Fellowship, awarded by National Board for Higher Mathematics, Department of Atomic Energy, India.

\section{Preliminaries on Rost's cycle complex over a regular base}
In this section we recall the basic notations from ~\cite{kato86, Rost} relevant to this paper.  
Let $S$ be a Noetherian regular scheme, hence an excellent scheme. All $S$-schemes throughout this paper are essentially finite type over $S$, hence are excellent schemes. 

\begin{definition}
Let $X$ denote a finite type $S$-scheme. Let $n\geq 0$. Rost's cycle complex for the Milnor K-theory $\KM_n$ on $X$ is defined as
\begin{equation*}
C^{\Dot}(X, \KM_n): C^0(X, \KM_n)\xrightarrow{d} C^1(X, \KM_n) \xrightarrow{d}\cdots \xrightarrow{d} C^{n-1}(X, \KM_n) \xrightarrow{d} C^n(X, \KM_n) \to 0
\end{equation*}
where 
\begin{equation*} 
C^{q}(X, \KM_n):=\bigoplus_{x\in X^{(q)}}\KM_{n-q}(\kappa(x))
\end{equation*}
 and $X^{(q)}$ is the set of codimension $q$ points of $X$. The differential $d$ is defined as follows:

Let $x\in X^{(q)}$, and $y\in X^{(q+1)}$. For $y\notin \-{\{x\}},$ define $d^x_y=0$. Let $y\in \-{\{x\}}$. Let $Z$ be the normalization of $\-{\{x\}}$. For every $z\in Z^{(1)}$ lying over $y$, we define 
\begin{equation*} d^x_y:=\sum_{\substack{z\in Z^{(1)}\\z\mapsto y}}c^z_y\circ d^x_z 
\end{equation*}
where 
\begin{equation*} d^x_z: \KM_*(\kappa(x))\to \KM_{*-1}(\kappa(z))
\end{equation*}
 is the residue map for Milnor K-theory associated to the discrete valuation ring $\sO_{Z, z}$ with fraction field $\kappa(x)$ and the residue field $\kappa(z)$, and 
\begin{equation*} c^z_y: \KM_*(\kappa(z))\to \KM_{*}(\kappa(y))
\end{equation*}
is the Norm morphism associated to the finite field extension $\kappa(y)\subset\kappa(z)$.
Define the differential 
\begin{equation*} 
d:=\sum_{\substack{(x,y)\\ y\in\-{\{x\}}}} d^x_y: C^q(X, \KM_n)\to C^{q+1}(X, \KM_n). \end{equation*}
We denote the $p^{th}$ cohomology of the complex $C^{\Dot}(X, \KM_n)$ by $A^p(X, \KM_n).$
\end{definition}
\begin{remark}

We recall the notion of cycle complex with support on a closed subscheme.
Let $Z\subset X$ be a closed subscheme of $X$ of codimension $c$ and $U:=X-Z\xrightarrow{j} X$ be the open complement of $Z$ in $X$. Define 
\begin{equation*} 
C^{\Dot}_{Z}(X, \KM_{n}):=\ker\Big(C^{\Dot}(X, \KM_n)\xrightarrow{j^*} C^{\Dot}(U, \KM_n)\Big).
\end{equation*}
We denote by $A^p_Z(X, M)$ the $p^{th}$ cohomology of the chain complex $C^{\Dot}_{Z}(X, \KM_{n}).$
\end{remark}

We also recall some basic properties of the cycle complexes.
\begin{definition}
\begin{enumerate}
\item
Let $f:X\to Y$ be a finite type morphism of equi-dimensional $S$-schemes with $\dim X=d$ and $\dim Y=e$.  Then there is a map 
\begin{equation*} 
f_*: C^q(X, \KM_n)\to C^{e-d+q}(Y, \KM_n)
\end{equation*}
 defined for 
 $\alpha\in \KM_{n-q}(\kappa(x))$ as
 \begin{align*}({f_*})_{xy}(\alpha)=
\begin{cases}
 c^{\kappa(x)}_{\kappa(y)}(\alpha) & \text{if $f(x)=y$ and $\kappa(y)\subset\kappa(x)$ is a finite field extension}\\
 0 & \text{otherwise}
 \end{cases}
 \end{align*}
 where  $c^{\kappa(x)}_{\kappa(y)}:\KM_*(\kappa(x))\to \KM_*(\kappa(y))$ is the Norm morphism for Milnor K-theory of fields.
 \item 
Let $g:Y\to X$ be a flat morphism of equi-dimensional $S$-schemes with constant relative dimension $s$. Then there is a map
\begin{equation*} 
g^*: C^q(X, \KM_n)\to C^{q}(Y, \KM_n)
\end{equation*}
defined for 
 $\alpha\in \KM_{n-q}(\kappa(x))$ as
  \begin{align*}({g^*})_{xy}(\alpha)=
\begin{cases}
 l(\sO_{Y_x, y})i_*(\alpha)      &\text{if $g(y)=x$ and $y$ is a generic point of $Y_x$}\\
 0    & \text{otherwise}
 \end{cases}
 \end{align*}
where $l(\sO_{Y_x, y})$ is the length of the local ring $\sO_{Y_x, y}$ of $Y_x$ at $y$ and $i_*: \KM_*(\kappa(x))\to \KM_*(\kappa(y))$ is the restriction map.
 
 \item For a unit $u\in \sO(X)$, there is a map 
\begin{equation*} 
\{a\}: C^q(X, \KM_n)\to C^{q}(X, \KM_{n+1})
\end{equation*}
 defined by 
\begin{align*}\{a\}_{xy}(\alpha)=
\begin{cases}
 \{a\}\cdot\alpha      &\text{if $x=y$ }\\
 0    & \text{otherwise}
 \end{cases}
 \end{align*}

 \item Let $i: Z\hookrightarrow X$ be a closed subscheme of $X$ and $U:=X-Z\xrightarrow{j} X$ be the open complement of $Z$ in $X$. There are term-wise maps  $j_*:C^q(U, \KM_n)\to C^q(X, \KM_n)$ and  $i^*: C^q(X, \KM_n)\to C^q_Z(X, \KM_n)$, then there is a map
\begin{equation*} 
\partial^U_Z:C^q(U, \KM_n)\to C^{q+1}_Z(X, \KM_n)
 \end{equation*}
 defined as 
\begin{equation*} 
\partial^U_Z:= i^*\circ d_X\circ j_*.
\end{equation*}
 \end{enumerate}
\end{definition}

We recall some properties satisfied by the above maps which will be used in the next section. The properties below are proved in~\cite[Proposition 4.6]{Rost} when the base is the spectrum of a field. When the base in the spectrum of a discrete valuation ring, the properties are proved in~\cite[Lemma 2.7]{lueders20}.
\begin{lemma} 

\noindent\begin{enumerate}
\item For a proper morphism $f:X\to Y$ of $S$-schemes, $f_*\circ d_X=d_Y\circ f_*$.
\item  Let $g:Y\to X$ be a flat morphism of equi-dimensional $S$-schemes with constant relative dimension $s$. Then $g^*\circ d_X=d_Y\circ g^*$.
\item For a unit $u\in \sO(X)$, $\{u\}\circ d_X=-d_X\circ \{u\}$. 
\end{enumerate}

\end{lemma}
\begin{proof} The proof as in~\cite[Lemma 2.7]{lueders20} applies varbatim.

\end{proof}

\section{Exactness of the Gersten complex for Milnor K-theory}
In this section we will further assume that the base $S$ is the spectrum of a Noetherian Henselian regular local domain and let $s$ denote the closed point of $S$.

\begin{lemma}(~\cite[Lemma 4.5]{Rost})
\label{boundary}
Let $g: Y\to X$ be a smooth $S$-morphism of finite type $S$-schemes of relative dimension 1. Let $\sigma:X\to Y$ be a section to g and let $t\in \mathcal{O}_Y(Y)$ be a global parameter defining the subscheme   $\sigma(X)$. Moreover, let $\tilde{g}:U=Y-\sigma(X)\to X$ be the restriction of g and let $\partial$ be the boundary map associated to the tuple $(X\xrightarrow{\sigma} Y \leftarrow U)$, then 
\begin{equation*}
\partial\circ\tilde{g}^*=0
\end{equation*}
and 
\begin{equation*}
\partial\circ \{t\} \circ \tilde{g}^*=(id_X)_*.
\end{equation*}
\end{lemma}

We also need the following version of ~\cite[Proposition 6.4]{Rost} for our relative setting.
\begin{proposition}
\label{locvanish}
Let $X$ be a smooth, irreducible, equi-dimensional $S$-scheme of relative dimension $d$. Let $Y\subset X$ be a closed subscheme of codimension $c\geq 1$ such that $\dim Y_s < \dim X_s$.
Let $y$ be a point in $Y$ lying over $s\in S$. Then there is a Nisnevich neighbourhood $X'$ of $y$ in $X$ such that the map
\begin{equation*}
A^p_{Y\times_X X'}(X', \KM_{n})\to A^p(X', \KM_n)
\end{equation*}
is trivial for any $n\geq 0$ and $p\geq 0$.
\end{proposition}

\begin{proof} We follow Rost's proof of Proposition 6.4 in~\cite{Rost} using the Lemma~\ref{boundary}.
By the assumption on $Y$, using~\cite[Theorem 1.1]{DHKS20} and \cite[Theorem 3.4]{DHKS21}, we have a Nisnevich neighbourhood $X'$ of $y$ such that, there is a morphism 
\begin{equation*}\pi: X'\to\A^{d-1}_S
\end{equation*}
of relative dimension 1 which is smooth at $y$ and $\pi|_{Y'}: Y':=Y\times_X X'\to \A^{d-1}_S$ is finite. Let $Z:=Y'\times _{ \A^{d-1}_S} X'$. Let $\sigma: Y'\to Z$ be the section of $q:Z\to  Y'$.
\begin{equation*}
\xymatrix{
Y' \ar[rd]^{\sigma}\ar[rdd]_{~Id_Y}\ar[rrd]^{i}&        &\\
   & Z    \ar[r]^{~p}\ar[d]^{q}  & X' \ar[d]_{\pi}\\
   & Y'    \ar[r]^{\pi|_{Y'}}  & \A^{d-1}_S\\
}
\end{equation*}
Passing to a small open neighbourhood of $X'$, we can assume that $\sigma(Y')\subset Z$ is given by a section $t\in \sO_{Z}(Z)$. Let $Q:=Z-\sigma(Y')$ and $j:Q\to Z$ be the open immersion. Consider the correspondence 
\begin{equation*}
H: Y'\xrightarrow{q^*} Q\xrightarrow{\{t\}} Q\xrightarrow{j_*} Z\xrightarrow{p_*} X'
\end{equation*}
giving a homomorphism 
\begin{equation*}
H: C^{p}_{Y'}(X', \KM_{n})\to C^{p}(X', \KM_{n+1}).
\end{equation*}
such that
\begin{equation}
\label{homM}
d_{X'}\circ H+H\circ d_{Y'}=i_*
\end{equation}
where $i:Y'\inj X'$ is the closed immersion. 
\begin{align*}
d_{X'}\circ H&=d_{X'}\circ p_*\circ j_*\circ \{t\}\circ q^*\\
                 &\overset{\ref{1}}= p_*\circ d_Z\circ j_*\circ \{t\} \circ q^*  \\
                 &\overset{\ref{2}}= p_*\circ(\sigma_*\circ \partial^Q_{Y'}+j_*\circ d_Q)\circ \{t\} \circ q^*\\
                 &= p_*\circ \sigma_*\circ \partial^Q_{Y'}\circ \{t\} \circ q^*+ p_*\circ j_*\circ d_Q\circ \{t\} \circ q^*\\
                 &\overset{\ref{3}}=p_*\circ \sigma_*\circ {id_{Y'}}_*+ p_*\circ j_*\circ (-\{t\}\circ d_Q) \circ q^*\\
                 &\overset{}=i_*+p_*\circ j_*\circ( -\{t\})\circ q^*\circ d_{Y'}\\
                 &=i_*-p_*\circ j_*\circ \{t\}\circ q^*\circ d_{Y'}\\
                 &=i_*-H\circ d_{Y'}.
\end{align*}
\begin{enumerate}[label=\textnormal{(\roman*)}]
\item \label{1} follows since $p$ is proper, hence $d_{X'}\circ p_*=p_*\circ d_Z$~\cite[Prop. 4.6 (1), page 352]{Rost}.
\item\label{2} follows since $\sigma_*\circ\sigma^*+j_*\circ j^*=1_Z$ \cite[(3.10), page 350]{Rost}, thus  
\begin{equation*}
\sigma_*\circ\sigma^*\circ d_Z\circ j_*+j_*\circ j^*\circ d_Z\circ j_*=d_Z\circ j_*.
\end{equation*}
 Since $\partial^Q_{Y'}=\sigma^*\circ d_Z\circ j_*$, we have
$d_Z\circ j_*=\sigma_*\circ\partial^Q_{Y'}+j_*\circ d_Q.$
\item \label{3}First term in \ref{3} follows by  Lemma~\ref{boundary} applied to triple $(Y'\xrightarrow{\sigma}Z\xleftarrow{j}Q).$ The second term in \ref{3} follows since $d_Q\circ \{t\}=-\{t\}\circ d_Q$ \cite[Prop. 4.6 (3), page 353]{Rost}.
\end{enumerate}
Thus, the relation~\eqref{homM} implies that $i_*: A^p_{Y\times_X X'}(X', \KM_{n})\to A^p(X', \KM_n)$ is zero.
\end{proof}

%
%
%

We now prove the main result of this paper, Theorem~\ref{mainthm} in the introduction.
\begin{theorem}\label{main}
Let $X$ be an essentially smooth Henselian local $S$-scheme,  then
\begin{equation*}
 A^p(\KM_n)(X)=0
 \end{equation*}
\text{for} $p\geq \dim S\geq 1$ and $n\geq 0.$
\end{theorem}

\begin{proof}The proof follows the ideas in~\cite[Proposition 6.1]{Rost} and~\cite[Theorem 3.1]{lueders20}) with input from Proposition~\ref{locvanish}.

Let $X=U_u^h:=\Spec \mathcal{O}_{U, u}^h$ where $U$ is a smooth, irreducible, equi-dimensional $S$-scheme of dimension $d$ and a point $u$ in $U$ lying over the closed point $s$ in $S$. We have
\begin{equation*}
C^p(X, \KM_n)=\dlim_{(U', u)} C^p(U',\KM_n)
\end{equation*}
and
\begin{equation*}
C^p(U',\KM_n)=\dlim_{Y\subset U'} C^{p}_Y(U', \KM_n)
\end{equation*}
where $Y$ runs over closed subschemes of $U'$ of dimension $d-p$.
 
We have 
\begin{equation*}
A^p(\KM_n)(X)=    \dlim_{(U', u)} A^p(\KM_n)(U')    
\end{equation*}
where $(U', u)$ runs over Nisnevich neighbourhoods of $u$ in $U$.
Thus, every element of  $A^p(U', \KM_n)$ is represented by an element in $C^{p}_Y(U', \KM_n)$
for some closed subscheme $Y$ of $U'$ of dimension $d-p$.

We claim that, every element of $A^{p}(U', \KM_n)$ is represented by an element of $C^{p}_Y(U', \KM_n)$ where $Y$ runs over all closed subschemes of $ U'$ of dimension $d-p$, such that $\dim Y_s<\dim U'_s$ (note that $Y_s\neq \emptyset$ as $u\in Y_s$). 

Let $y\in Y^{(0)}$ such that $\dim \-{\{y\}}_s=\dim U'_s$. If $y\notin U'_s$, then 
\begin{equation*}
d-\dim S=\dim U'_s=\dim \-{\{y\}}_s\leq \dim\-{\{y\}}-1= d-p-1
\end{equation*}
which is impossible, by the assumption on $p$ that $p\geq \dim S$. So we can assume that $y\in U'_s$.

By  \cite[Lemma 7.2, page 1622]{SS} applied to the tuple $(y\in U'_s\subset U')$, 
there exists an integral closed subscheme $Z\subset U'$ of codimension $p-1$  such that 
\begin{enumerate}
\item $ Z\cap (U'- U'_s)\neq \emptyset$
\item $y\in Z$
\item $Z$ is regular at $y$.
\end{enumerate}
By the condition (1) above, $Z-U'_s\subset Z$ is a non-empty, dense, open subset of $Z$. We partition the codimension 1 points of $Z$ as $Z^{(1)}=Z_s^{(1)}\coprod Z_{\eta}^{(1)}$ where
\begin{equation*}
Z_{\eta}^{(1)}:=\Bigl\{\-{\{w\}} ~\text{codimension 1 in $Z$ such that}~ w\in Z-U'_s 
\Bigr\}
\end{equation*}
and 
\begin{equation*}
Z_s^{(1)}:=Z^{(1)}-Z_{\eta}^{(1)}=\Bigl\{\-{\{w\}} ~\text{codimension 1 in $Z$ such that}~ w\in Z\cap U'_s
\Bigr\}.
\end{equation*}
We observe that $\-{\{y\}}\in Z_s^{(1)}$. We claim that $Z_s^{(1)}$ is a finite set and $\-{\{w\}}\in Z_{\eta}^{(1)}$ implies that $\dim \-{\{w\}}_s<\dim U'_s$.

To see the claim, first we have by assumption that $\dim (\-{\{y\}})_s=\dim U'_s$. (This implies that a maximal component of $\-{\{y\}}_s$ coincides with a maximal component of $U'_s$.) 

By earlier discussion, we have $y\in U'_s$, so that  $\-{\{y\}}\subset U'_s$. Let $w\in Z^{(1)}$ such that $w\in Z\cap U'_s$. Then $d-p=\dim \-{\{w\}}\leq\dim Z\cap U'_s\leq d-p.$ Thus, $\-{\{w\}}$ is a component of  $Z\cap U'_s$. Since $Z\cap U'_s$ has finitely many components, $Z_s^{(1)}$ is a finite set. We write $Z_s^{(1)}=\{z_1, \dots, z_m, y\}$.

Further, we observe that  $\-{\{w\}}\in Z_{\eta}^{(1)}$ implies that $\dim \-{\{w\}}_s<\dim U'_s$. 
Otherwise, let $\-{\{w\}}\in Z_{\eta}^{(1)}$ be such that $\dim \-{\{w\}}_s=\dim U'_s$. Then
$\dim \-{\{w\}}_s\leq \dim \-{\{w\}}=d-p=\dim U'_s$. This implies that  $\dim \-{\{w\}}_s= \dim \-{\{w\}}$. Hence $w\in \-{\{w\}}_s\subseteq U'_s$, which is impossible since $w\notin U'_s$. 


Now, let $\{\-\alpha_1,\-\alpha_2, \dots, \-\alpha_n\}\in \KM_n(\kappa(y)).$ Let $\fr p_i$ (resp. $\fr p_y$) be the maximal ideals corresponding to $z_i$ (resp. $\-{\{y\}}$) in the semi-local ring $\sO_{Z, Z^{(1)}_s}$. 
Since the ideals $\fr p_1,\dots, \fr p_m, \fr p_y$ are mutually coprime in $\sO_{Z, Z^{(1)}_s}$, the canonical map  
\begin{equation*}
\sO_{Z, Z^{(1)}_s}\to \sO_{Z, Z^{(1)}_s}/\fr p_y\times \prod_{i=1}^m \sO_{Z, Z^{(1)}_s}/\fr p_i
\end{equation*}
 is surjective. Let $\alpha_i\in \sO_{Z, Z^{(1)}_s}$ such that $\alpha_i\mapsto \-\alpha_i$ in $k(\-{\{y\}})^*$ and 1 in  $k(z_i)^*$.
Let $\pi$ be a local parameter of $Z$ at $y$. Then under the differential
\begin{equation*}
d=(d^z_w)_{w\in Z^{(1)}}: \KM_{n+1}(k(z))\to \Bigg(\bigoplus_{i=1}^m \KM_{n}(k(z_i))\Bigg) \oplus \KM_{n}(k(y))\oplus \Bigg(\bigoplus_{w\in Z_{\eta}^{(1)}} \KM_{n}(k(w))\Bigg)
\end{equation*}
\begin{equation*}
d(\{\alpha_1,\alpha_2, \dots, \alpha_n, \pi\})\mapsto (0+ \dots+ 0+\{\-\alpha_1,\-\alpha_2, \dots, \-\alpha_n\})+ \sum_{w\in Z_{\eta}^{(1)}} d^{z}_w(\{\alpha_1,\alpha_2, \dots, \alpha_n, \pi\}).
\end{equation*}
We observe that $\sum_{w\in Z_{\eta}^{(1)}} d^{z}_w(\{\alpha_1,\alpha_2, \dots, \alpha_n, \pi\})\in C^{p}(U', \KM_n)$ such that $W=\-{\{w\}}$ is a codimension $p$ closed subscheme such that $\dim W_s<\dim U'_s$.
Thus, we have that any element of $A^{p}(U', \KM_n)$ is represented by 
an element in $A^{p}_Y(U', \KM_n)$
for some closed subscheme $Y\subset U'$ of codimension $p$ such that $\dim Y_s<\dim U'_s$. 
Now by Proposition~\ref{locvanish}, for some $(U'', z)$ a Nisnevich neighbourhood of $z$ in $U'$, 
\begin{equation*}
 A^{p}_{Y\times_{U'} U''}(U'', \KM_n)\to A^{p}(U'', \KM_n)
 \end{equation*}
is zero map. 

Thus, the map  
\begin{equation*}
A^{p}_Y(U', \KM_n)\to A^{p}_{Y\times_{U'} U''}(U'', \KM_n)\to A^{p}(U'', \KM_n)
\end{equation*}
is zero implies that
\begin{equation*}
A^{p}_Y(U', \KM_n)\to   A^p(U', \KM_n)\to A^{p}(U'', \KM_n)
\end{equation*}
is the zero morphism.

This shows that 
\begin{equation*}
A^p(\KM_n)(X)=    \dlim_{(U', u)} A^p(\KM_n)(U') =0.
\end{equation*}
\end{proof}

\begin{remark}
It is worthwhile, to point out that the above proof fails for degrees $p<\dim S$,  since for example, $Z$ constructed as in the proof above, can have infinitely many codimension 1 points containing the closed fiber $U'_s$, which makes the set $Z_s^{(1)}$ infinite and the application of the Chinese remainder theorem impossible. More concretely, 
consider $S=\Spec\Z_l\llbracket x\rrbracket$ for a prime $l$, an indeterminate $x$ and the ring of $l$-adic integers $\Z_l$. Take  $Z=U'=\A^1_S=\Spec \Z_l\llbracket x\rrbracket[t]$. Now for each $a\in \Z_l$, $Z$ has codimension 1 point $z_a$ associated to the ideal $(l+ax),$ so that $\-{\{z_a\}}$ contains the closed fiber $U'_s$ which is given by the ideal $(l, x)$.
\end{remark}

\end{document}